\documentclass[11pt]{article} 
\usepackage[utf8]{inputenc}
\usepackage[T1]{fontenc}
\usepackage[twoside,left=3.66cm,right=3.66cm,bottom=3cm,top=2.66cm]{geometry}
\usepackage{microtype}
\usepackage{amsmath,amsthm,amsfonts,amssymb}
\usepackage{mathrsfs}
\usepackage{cite}
\usepackage{enumerate}
\usepackage{cite}
\usepackage[all,2cell]{xy} \SilentMatrices
\usepackage{pdftricks}
\usepackage{epsfig,color}
\usepackage{mathtools}
\usepackage{tikz-cd}
\usepackage[french,english]{babel}

\usepackage{framed}
\usepackage{color}
\definecolor{shadecolor}{gray}{0.875}
\definecolor{col}{RGB}{42, 95, 151}
\usepackage[colorlinks=true, linkcolor=col, citecolor=col, filecolor = col, menucolor = col, urlcolor = col]{hyperref}

\theoremstyle{plain}
\newtheorem{theorem}{Theorem}[section]
\newtheorem*{lemma*}{Lemma}
\newtheorem{lemma}[theorem]{Lemma}
\newtheorem*{theorem*}{Theorem}
\newtheorem{proposition}[theorem]{Proposition}
\newtheorem*{proposition*}{Proposition}
\newtheorem{corollary}[theorem]{Corollary}
\newtheorem*{corollary*}{Corollary}

\theoremstyle{definition}
\newtheorem{remark}[theorem]{Remark}
\newtheorem*{remark*}{Remark}
\newtheorem*{definition*}{Definition}

\newtheorem*{example*}{Example}

\newtheorem{question}{Question}
\newtheorem*{question*}{Question}

\def\c1{\operatorname{c_1}}
\def\c2{\operatorname{c_2}}
\def\Spec{\operatorname{Spec}}

\def\CC{{\mathbb C}}
\def\ZZ{{\mathbb Z}}

\def\QQ{{\mathbb Q}}
\def\PP{{\mathbb P}}

\def\CP{{\mathbb P}}
\def\L{{  L}}

\def\O{{\mathcal O}}

\def\AA{{\mathbb A}}

\def\+{\oplus}                   
\def\*{\otimes}

\def\Pic{\operatorname{Pic}}

\def\CH{\operatorname{CH}}

\newcommand\blfootnote[1]{%
  \begingroup
  \renewcommand\thefootnote{}\footnote{#1}%
  \addtocounter{footnote}{-1}%
  \endgroup
}

\def\L{\O_X(1)}
\def\X{\mathcal X}

\def\Kollar{Koll\'ar }

\title{On deformations of quintic and septic hypersurfaces} 
\author{John Christian Ottem and Stefan Schreieder}

\date{ }

\begin{document}

\maketitle 

\thispagestyle{empty}

\blfootnote{{\it Date}: November 6, 2019}

In this paper we will be interested in a question of Mori \cite[p.\ 642]{mori} from 1975, which in recent years has also been advertised by Koll\'ar (see e.g.\ \cite[p.\ 33]{kollar-moduli}) (see also the Mathoverflow post \cite{MO}).
To state the question, we say that a projective variety $X$ is a smooth specialization of a hypersurface of degree $d$ if there is a smooth proper morphism $\mathcal X\to \Spec A$ over a local ring $A$, with special fibre $X$, whose generic fibre is isomorphic to a hypersurface of degree $d$ in projective space.

\begin{question}\label{question}
Let $X$ be a smooth specialization of a hypersurface of prime degree in $\CP^{n+1}$ with $n\geq 3$.
Is then also $X$ a hypersurface? 
\end{question}

In dimensions 1 and 2 there are many counterexamples to this question (even in the case where $X$ is defined over $\CC$). 
There are explicit families of smooth plane curves which specialize to hyperelliptic curves (which are not embeddable in $\PP^2$) \cite{griffin}. In dimension 2, we can view cubic surfaces as blow-ups of $\PP^2$ in six general points and specialize to a non-cubic surface by moving the points into special position.   
Moreover, simultaneous resolution of ADE surface singularities \cite{artin} 
shows that the simultaneous resolution of a family of smooth surfaces in $\PP^3$ degenerating to a surface with only ADE singularities gives a counterexample for any degree $\ge 2$ (as long as we allow the total space $\mathcal X$ to be an algebraic space).   
A much more sophisticated example was found by Horikawa \cite{horikawa}, who showed that over the complex numbers, quintic surfaces specialize to smooth surfaces which themselves, as well as their canonical models, cannot be embedded into $\CP^3$.

The restriction to prime degrees in Question \ref{question} is also necessary. Mori \cite{mori} constructed the following examples, showing that for any composite degree and in arbitrary dimension $\geq 3$, hypersurfaces can deform to non-hypersurfaces: if $f$ and $g$ are generic homogeneous polynomials in $x_0,\ldots, x_n$ of degree $ab$ and $b$ respectively, the family
$$
\X=Z(y^a-f(x_0,\ldots,x_n),ty-g(x_0,\ldots,x_n))\subset \PP(1^{n+1},b)\times \AA^1_t
$$
is smooth, and $\X_t$ is a degree $ab$ hypersurface for $t\neq 0$, but $\X_0=Z(y^a-f,g)$ is not isomorphic to a hypersurface: it is a $a:1$ cover of the hypersurface $Z(g)\subset \PP^n$ and so the ample generator of the Picard group is not very ample. 
Replacing $t$ by a uniformizer of $\ZZ_p$, 
we obtain similar examples for specializations in mixed characteristic.

In light of the above examples, one could perhaps expect counterexamples to Question \ref{question}, but coming up with prime degree examples makes it more difficult.  
After all, 
 if $X$ is defined over an algebraically closed 
field of characteristic zero, the result of Brieskorn \cite{brieskorn} (see also the related results in \cite{kobayashi-ochiai} and \cite{hwang}) and Fujita \cite{fujita,fujita-book} imply a positive answer to Question \ref{question} in degrees $2$ and $3$, respectively.

In this paper we settle Mori's question for quintics in arbitrary dimension and for septics in dimension three. 

\begin{theorem} \label{thm:question1}
Let $X$ be a smooth projective variety of dimension $n\geq 3$ over a field $k$.
Assume that $X$ is the smooth specialization of a hypersurface of degree $d$, and that one of the following holds:
\begin{enumerate}[(a)]
\item $(n,d)=(3,7)$; 
\label{item:cor:question1:1}
\item $(n,d)=(3,5)$; \label{item:cor:question1:0}
\item  $n\geq 4$, $d=5$ and $\operatorname{char} k=0$. \label{item:cor:question1:2}
\end{enumerate} 
Then $X$ is isomorphic to a hypersurface of degree $d$ in $\CP_{k}^{n+1}$. 
\end{theorem}

The examples in items (\ref{item:cor:question1:1}), (\ref{item:cor:question1:0}) and (\ref{item:cor:question1:2}) are of general type, Calabi--Yau and Fano, respectively.  
Note also that the above theorem classifies arbitrary smooth specializations of quintic or septic threefolds over local rings of mixed characteristic.

 Item (\ref{item:cor:question1:2}) 
 in Theorem \ref{thm:question1} will be deduced from the following numerical characterization of smooth quintic hypersurfaces in dimension at least three. 
 
\begin{theorem}\label{quintictheorem}
Let $X$ be a smooth projective variety of dimension $n\geq 3$, defined over an algebraically closed field $k$ of characteristic zero.
Suppose that $\Pic X/\sim_{num}=\ZZ [\L]$ for an ample line bundle $\L$ with 
$$
\L^n=5,\ \ \chi(X,\L)\geq n+2\ \ \text{and}\ \ K_X=\mathcal O_X(3-n).
$$
Then $X$ is isomorphic to a quintic hypersurface in $\CP^{n+1}_k$.
\end{theorem}

Items (\ref{item:cor:question1:1}) and (\ref{item:cor:question1:0}) in Theorem \ref{thm:question1} will be deduced from numerical characterizations of quintic and septic threefolds in arbitrary characteristic that are similar to Theorem \ref{quintictheorem} above, see Theorems \ref{thm:n=3,d=5} and \ref{thm:n=3,d=7} below.

Theorem \ref{quintictheorem} has the following consequence.

\begin{corollary}\label{corprojdef}
Let  $X$ be a smooth complex projective variety of dimension $n\geq 3$ which deforms via a sequence of K\"ahler deformations 
to a quintic hypersurface in $\CP^{n+1}_\CC$.
Then $X$ is isomorphic to a quintic hypersurface in $\CP^{n+1}_\CC$.
\end{corollary}

The above corollary  
extends the famous work of Horikawa \cite{horikawa} on deformations of quintic surfaces.
In fact, Horikawa  
showed that any minimal complex algebraic 
surface $X$ with invariants $p_g=4$, $q=0$ and $K_X^2=5$ is deformation equivalent to a quintic hypersurface. 
Moreover, there are two cases for such a surface: either $K_{X}$ is base point free, in which case it gives a morphism onto the canonical model of $X$ which is a quintic surface; 
or $K_X$ has a unique base point, and the rational map $\phi_{|K_X|}:X\dashrightarrow F$ is generically a double cover of a Hirzebruch surface. 
Griffin \cite{griffin} extended this analysis by writing down explicit equations for a smooth family of quintic surfaces specializing to a surface of the latter kind (we review this construction in Section \ref{sec:reid}).

It is not hard to see that Question \ref{question} has a positive answer if the specialization  $\mathcal O_X(1)$  of the hyperplane bundle 
is base point free, see Lemma \ref{lem:H^n=d} below. 
However, proving this directly from general theorems using cohomological techniques or vanishing theorems does not seem to be possible. 
For instance, Reid \cite{reid} constructed an example of a specialization $X$ of quintic threefolds such that  
$\mathcal O_X(1)$ 
has an isolated base point (hence $X$ is not a hypersurface), but where $X$ has terminal singularities (two ordinary double points), see Section \ref{sec:reid} below. 

The analogue of Question \ref{question} for hypersurfaces in abelian varieties is known to have a positive answer, see \cite{CL}.
Note however that this case is much simpler, as one has the Albanese morphism to start with.
In particular, the main difficulties that we face in this paper, namely the possibility of base points of $\mathcal O_X(1)$, disappear.

While we believe that Question \ref{question} is interesting in its own right, understanding specializations of hypersurfaces is known to be an important tool in the study of the geometry of hypersurfaces themselves. 
For instance, Koll\'ar \cite{K-JAMS} and Totaro \cite{totaro-JAMS} used the aforementioned deformations of Mori (in mixed characteristic) to show that many hypersurfaces of low degree are not ruled, respectively not stably rational.
Totaro's result has recently been improved in \cite{Sch1}, where different degenerations are used.

\bigskip

\noindent {\bf Acknowledgements.}  
We would like to thank J.\ \Kollar and B.\ Totaro, for bringing this problem to our attention. We would also like to thank B.\ Totaro for comments and O.\ Benoist, R.\ Laza, J. V.\ Rennemo, and C.\ Xu for useful discussions.

\section{Basic set-up and some general results} 

\subsection{Notation} \label{subsec:notation} 

Let $k$ be a field, which until Section \ref{sec:question1} will be assumed to be algebraically closed.
We denote by $X$ a smooth  projective variety of dimension $n$ over $k$, whose Picard group is up to numerical equivalence generated by an ample line bundle $\O_X(1)$ with 
$$
\L^n=d,\ \ h^0(X,\L)\geq n+2\ \ \text{and}\ \ K_X=\mathcal O_X(d-n-2)
$$
for some integer $d\geq 1$.
If $\operatorname{char} k=0$, $n\geq 3$ and $d=5$, then Kodaira vanishing shows $\chi(X,\L)=h^0(X,\L)$ and so the above assumptions are exactly those from Theorem \ref{quintictheorem}.

Let us fix some notation, which will be retained for the rest of the paper. 
We will consider the linear system $|\O_X(1)|$ with base locus $\mbox{Bs}(|\O_X(1)|)\subset X$.  
We denote by $B$ the cycle of top-dimensional components of $\mbox{Bs}(|\O_X(1)|)$, counted with multiplicities. 
We further denote the dimension of $B$ by $b=\dim B$.

From now on we assume that 
we can resolve the rational map $\phi$ given by $n+2$ general sections of $\L$, and get a diagram
\begin{equation} \label{eq:diagr}
\begin{tikzcd}
 & W \arrow[ld, "p"'] \arrow[rd, "q"] &  \\
X \arrow[rr, "\phi", dashed] &  & Y \subset \PP^{n+1}
\end{tikzcd}
\end{equation}
where $W$ is smooth, $p$ is an isomorphism away from the base locus of $\L$ and $Y=\overline{\phi(X)}$ is integral. 
By Hironaka's theorem, this assumption is satisfied if $k$ has characterstic zero.
By \cite[p.\ 1839]{CP}, it is also satisfied over arbitrary algebraically closed fields $k$ if $n=3$. 

We will use the following notation for divisors: we will let $L=p^*\O_X(1)$; by the construction of $W$, we can decompose this as 
$$
L=M+F,
$$ 
where $M=q^\ast \mathcal O_Y(1)$ is the movable part of $L$ and $F$ is the fixed part.
It follows from Lemma \ref{lem:B,degY} below that $F$ is supported on the $p$-exceptional locus and so  
$p_*M=\L$.

\subsection{Some useful lemmas}
We will repeatedly make use of the following lemmas.

\begin{lemma}\label{lem:B,degY}
In the notation of Section \ref{subsec:notation},  $\dim B\leq n-2$ and $\deg Y\geq n+2-\dim Y$.
\end{lemma}
\begin{proof}
Since $\Pic X$ is up to numerical equivalence generated by $\L$, and $h^0(X,\O_X(1))\ge 2$, the base locus of $\L$ cannot contain any divisorial components and so $\dim B\leq n-2$.
Moreover, $\deg Y\geq n+2-\dim Y$ because $Y$ is integral and not contained in any hyperplane of $\CP^{n+1}$ \cite{EH}.
\end{proof}

\begin{lemma}\label{lem:L^bM^n-b-1F}
In the notation of Section \ref{subsec:notation}, 
$L^i M^{n-i-1}F\geq 0$ for all $i$.
Moreover, if $\L$ is not base point free and $b=\dim B$, then $L^b M^{n-b-1}F>0$.
\end{lemma}

\begin{proof}
Since $\mathcal O_X(1)$ is ample, some multiple of $L$ is base point free.
Since $M$ is base point free as well and $F$ is effective, $L^iM^{n-i-1}F\geq 0$ for all $i$.

Let us now assume that $\L$ is not base point free and let $b=\dim B$.
Since $\phi$ is given by $n+2$ general sections of $\L$, the proper transform of a general element of $|\L|$ is given by a general element of $|M|$.
Using that $F$ dominates the base locus of $\L$, $\L$ is ample on $X$ and $L=p^\ast \L$, we then find that $M^{n-b-1} L^bF=0$ implies that an intersection of $n-b-1$ general elements of $|\L|$ on $X$ does not contain the base locus, a contradiction. 
This proves the lemma.
\end{proof}

\begin{lemma} \label{lem:decomp}
In the notation of Section \ref{subsec:notation}, $Y$ has no hyperplane section which is numerically equivalent to $H_1+H_2+D$ with $H_1,H_2$ effective movable divisors and $D$ effective.

In particular, $Y$ is not covered by linear spaces of codimension 1.
\end{lemma}

\begin{proof}
For a contradiction, suppose there is a hyperplane section $H$ of $Y$ that is numerically equivalent to $H_1+H_2+D$ with $H_1,H_2$ effective and movable (in the sense that they can be deformed to pass through general points of $Y$).
Then up to numerical equivalence, $q^*H$ is a Cartier divisor on $W$ with at least two movable components, and hence so is $p_*q^*H$. 
However, this divisor lies in the linear system $|\L|$, which contradicts the fact that $\L$ generates $\Pic X/\sim_{num}$ freely.

For the last statement, we assume $Y$ contains a linear subspace $H_1=\PP^{\dim Y-1}$ of codimension one whose deformations cover $Y$. 
Taking hyperplane sections containing $H_1$, we obtain a desired decomposition $H=H_1+H_2$ where $H_2$ is the residual effective divisor.
Since $Y\subset \CP^{n+1}$ is not contained in any linear subspace, it has degree at least two and so $H_2$ is nonzero.
For each $H_1$ there is at least a 1-dimensional family of hyperplanes in $\PP^{n+1}$ that contain it, and so we get a decomposition as above.
\end{proof}

\begin{lemma} \label{lem:dimY=1}
In the notation of Section \ref{subsec:notation}, assume that $n\geq 2$.
Then $\dim Y\geq 2$.
\end{lemma}
\begin{proof}
Since $Y$ is not contained in a hyperplane of $\CP^{n+1}$, it must be positive-dimensional.
The lemma thus follows from Lemma \ref{lem:decomp}, which shows that $\dim Y=1$ is impossible.
\end{proof}

\subsection{The case when $\phi$ is a morphism}

\begin{lemma} \label{lem:H^n=d}
In the notation of Section \ref{subsec:notation}, suppose that $d$ is a prime number.
If $\L$ is base point free, i.e.\ $\phi$ is a morphism, then $X$ is isomorphic to a hypersurface of degree $d$ and $\phi$ is an isomorphism.
\end{lemma}

\begin{proof}
Since $\L$ is base point free, $L=M$ and $\phi$ is a morphism.
Hence,  
$$
d=L^n=M^n=\deg Y\cdot \deg \phi .
$$
Since $d$ is prime and $\deg Y\ge 2$, we conclude that $Y$ is a hypersurface of degree $d$ and $\deg \phi=1$, i.e.\ $\phi$ is a birational morphism.

Let $C$ denote the complete intersection of $n-1$ elements of $\phi^\ast|\mathcal O_Y(1)|$. 
Then $C$ is a smooth curve such that 
$$
K_{C}=\O_X(-n-2+d+(n-1))|_C=\O_X(d-3)|_C,
$$
which has degree $d(d-3)$. 
Note that $C$ maps onto a complete intersection curve $C'\subset Y$.
Since $Y$ is a degree $d$ hypersurface, $C'$ is a plane curve of degree $d$. 
In particular, $C'$ is Gorenstein, and $K_{C'}=\O_Y((d-3)H)|_{C'}$ has degree $d(d-3)$ as well. 
It follows that $C\simeq C'$ is isomorphic to a plane curve of degree $d$. 
In particular, $Y$ is non-singular in codimension 2. 
Since $Y$ is a hypersurface in a smooth variety, it is also Cohen--Macaulay, and hence it is normal. 
Then since $\phi: X\to Y$ is finite of degree one, it is an isomorphism.
In particular, $X$ is isomorphic to a hypersurface of degree $d$, which concludes the lemma.
\end{proof}


\subsection{Bounds on $M^n$ and $\deg Y$}
The following generalizes a result of Horikawa \cite[Lemma 2]{horikawa}. 

\begin{lemma} \label{lem:Hodgeindex}
In the notation of Section \ref{subsec:notation}, we have the following
\begin{enumerate}[(a)]
\item $M^n\leq d$ and $n+2-\dim Y \leq \deg Y\leq d$; \label{item:degY}
\item  
if $M^n= d$, then $\L$ is base point free; \label{item:M^n=d}
\item 
if $M^n=d-1$, then $\L$ has exactly one base point.
If furthermore $\operatorname{char} k=0$, then the general element in $|\L|$ is smooth. \label{item:M^n=d-1}
\end{enumerate} 
\end{lemma}

\begin{proof} 
By Lemma \ref{lem:L^bM^n-b-1F}, $L^iM^{n-i-1}F\geq 0$ for all $i$.
Using $L=M+F$, this shows that
\begin{equation}\label{ineqs}
d=L^n\ge L^{n-1}M\ge L^{n-2}M^2\ge \cdots \ge M^n .
\end{equation}
In particular, $M^n\leq d$.
Moreover, $M^n=d$ implies $L^iM^{n-i-1}F=0$ for all $i$ and so $\L$ is base point free by Lemma \ref{lem:L^bM^n-b-1F}.
This proves item (\ref{item:M^n=d}) and the first part of item (\ref{item:degY}).
To prove the second part of item (\ref{item:degY}), let $y=\dim Y$ and let $C$ be a general fibre of $q$.
Then $M^y=\deg Y\cdot C$ in $\CH^y(W)$.
By (\ref{ineqs}), $d\geq L^{n-y}M^y$ and so the projection formula shows
$$
d\geq p_\ast( L^{n-y}M^y)=\deg Y \cdot \mathcal O_X(1)^{n-y}\cdot p_\ast C .
$$
Since $\L$ is ample on $X$ and $p_\ast C\neq 0$ because $C$ is a general fibre of $q$, $\L^{n-y}p_\ast C >0$ and so we conclude $\deg Y\leq d$.
The lower bound on $\deg Y$ follows from Lemma \ref{lem:B,degY}, which concludes the proof of item (\ref{item:degY}).

It remains to treat the case $M^n= d-1$. 
If $d=1$, then for any $i\leq n$, the intersection of general elements $D_1,\dots ,D_i\in |\L|$ must be irreducible and reduced, and of codimension $i$.
This shows that $\dim B=0$.
As a consequence, $p_\ast q^\ast (H_1\cdots H_{\dim Y})$ is irreducible and reduced for any general elements $H_i\in |\mathcal O_Y(1)|$.
This is a contradiction, because $\deg Y\geq 2$ by Lemma \ref{lem:B,degY}.

We may thus assume $d\geq 2$ and so $M^n=d-1$ shows that $\dim Y=n$, i.e.\ $M$ is big and nef.
Since $M^n=d-1$, the above inequality implies that among the terms $L^i M^{n-i-1}F$ for $i=0,\ldots,n-1$, all are in $\{0,1\}$ and exactly one equals 1. 
Since $M^n=d-1$, $\L$ is not base point free and we let $b=\dim B$ be the dimension of the base locus of $\L$.
By Lemma \ref{lem:L^bM^n-b-1F}, $L^bM^{n-b-1}F>0$ and so we must have $L^b M^{n-b-1}F=1$.

If $b>0$, then it follows that $L^{b-1}M^{n-b}F=0$ and so
$$
1=L^b M^{n-b-1}F=L^{b-1}M^{n-b}F+L^{b-1}M^{n-b-1}F^2=L^{b-1}M^{n-b-1}F^2 .
$$ 
This contradicts the Hodge index theorem (see e.g.\ \cite[Theorem 10.1]{KK}) 
applied to the resolution of an integral surface $S\subset W$ with class $(mL)^{b-1}M^{n-b-1}$ for some $m\gg 0$, because $M|_S$ is base point free and big, and $F|_S$ a non-zero effective divisor. 
Hence $b=0$. 

Since $b=0$ and $\phi$ is given by general sections of $\L$, the intersection of $n-1$ general elements of $|M|$ maps down to an intersection of $n-1$ general elements of $|\L|$, and so it must contain the base locus.
Therefore, $M^{n-1}F=1$ implies that the base locus of $\L$ is given by only one point $x\in X$. 
Moreover, the coefficient of the exceptional divisor in $F$ that meets the intersection of $n-1$ general elements of $|M|$ is $1$ and so a generic element $D\in |\L|$ has multiplicity one at $x$.
If additionally $\operatorname{char}k=0$, then it follows that $D$ is smooth, as it is smooth away from the base locus by Bertini's theorem in characteristic zero.
This concludes the proof of the lemma.
\end{proof}

\subsection{The case when $\phi$ is birational} \label{phibirational}

\begin{proposition} \label{prop:degq=1} 
In the notation of Section \ref{subsec:notation}, suppose that $d\geq n+2$.
If $\phi:X\dashrightarrow Y$ is birational, then $\phi$ is an isomorphism.
\end{proposition}

\begin{proof} 
Recall the diagram
$$
\begin{tikzcd}
 & W \arrow[ld, "p"'] \arrow[rd, "q"] &  \\
X  \arrow[rr, "\phi", dashed]&  & Y \subset \PP^{n+1}
\end{tikzcd} 
$$ 
We need to show that $\deg(q)=1$ implies that $\phi$ is an isomorphism.

Let $C\subset W$ be the intersection of $n-1$ general elements of $|M|$.
Since $\deg(q)=1$, $C':=q(C)$ is a plane curve of degree $d':=M^n$ and so $C'$ is Gorenstein with
$$
\deg(K_{C'})=(-3+d')d' .
$$
On the other hand,
$$
K_{C}=(K_W+(n-1)M)|_C ,
$$
where $K_W=(-n-2+d)(M+F)+\sum a_iE_i$ with $a_i\geq 1$, where $E_i$ denote the $p$-exceptional divisors of $W$.
We then get
$$
\deg(K_C)=(-3+d)M^{n}+(-n-2+d)F\cdot M^{n-1}+\sum a_iE_i\cdot M^{n-1}.
$$
Since $C$ is the normalization of $C'$, $\deg(K_{C})\leq \deg(K_{C'})$ and so
\begin{equation}\label{eq:ineq:degK}
(-3+d)M^{n}+(-n-2+d)F\cdot M^{n-1}+\sum a_iE_i\cdot M^{n-1}\leq (-3+d')d' .
\end{equation}
On the other hand, $F\cdot M^{n-1}\geq 0$ and $E_i\cdot M^{n-1}\geq 0$ for all $i$ because $M$ is base point free.
Moreover, $M^n=d'\leq d$ by Lemma \ref{lem:Hodgeindex}.
Since $-n-2+d$ is non-negative by assumption and $a_i\geq 1$ for all $i$, (\ref{eq:ineq:degK}) thus implies $(-n-2+d)F\cdot M^{n-1}=0$ and $E_i\cdot M^{n-1}=0$ for all $i$.
Hence,  
$$
(-3+d')d'\geq (-3+d)M^{n}=(-3+d)d'.
$$
This implies $d=d'$, because $d\geq d'$ by Lemma \ref{lem:Hodgeindex}.
Hence, $\phi$ is a morphism by Lemma \ref{lem:Hodgeindex} and so it is an isomorphism by Lemma \ref{lem:H^n=d}.
This concludes the proof of the proposition. 
\end{proof}

\section{Reduction to lower-dimensional cases for quintics}

 The main reduction in the proof of Theorem \ref{quintictheorem} is the following result, due to Mori, which allows us to restrict to low-dimensional cases provided we can find smooth hyperplane sections in $|\L|$, see \cite[Theorem 3.6]{mori}.
 
\begin{theorem}[Mori]\label{hyperplanesection}
Let $X$ be a normal projective variety of dimension $n\ge 3$ over a field $k$, equipped with an ample line bundle $\mathcal O_X(1)$.
Let $Z\in |\O_X(1)|$ and suppose that the global sections of $\mathcal O_X(1)|_Z$ embed $Z$ as a hypersurface in $\CP^n_k$.
Then $X$ is isomorphic to a hypersurface in $\CP_k^{n+1}$ 
such that $\L$ corresponds to the pullback of $\mathcal O_{\CP^{n+1}_k}(1)$. 
\end{theorem}

As shown in Horikawa's examples \cite{horikawa}, even in characteristic zero it is not true that under our given assumptions  the linear system $|\L|$ is base point free, so we cannot apply Proposition \ref{hyperplanesection} directly. 
In particular, it is not clear that $|\L|$ contains smooth elements. 
However, we at least have the following result of H\"oring--Voisin \cite[Theorem 1.6]{HV}, which builds on work of Kawamata \cite{kawamata} and Floris \cite{floris}.

\begin{theorem}[H\"oring--Voisin]\label{HVthm}
Let $k$ be an algebraically closed field of characteristic zero.
Let $X$ be a smooth Fano variety of dimension $n$ and index $n-3$  over $k$, and let $\L$ be the ample line bundle such that $-K_X=\mathcal O_X(n-3)$. 
Suppose further that $h^0(X,\L)\geq n-2$
and let $Z$ be a general intersection of $n-3$ divisors of $|\L|$. 
Then $Z$ is a threefold with isolated canonical singularities.
\end{theorem} 

Note however that in the above result, $Z$ might fail to be $\QQ$-factorial (see \cite[Example 2.12]{HV} for an example when $n=4$).
Even though we will not use this here, note that the assumption $h^0(X,\L)\geq n-2$ is automatically satisfied in Theorem \ref{HVthm}, see \cite{liu}.

The main observation of this section is the following.

\begin{proposition} \label{prop:HV}
In the notation of Theorem \ref{HVthm}, assume that $n \geq 5$.
Then a general element $D\in |\L|$ is smooth.
\end{proposition}

\begin{proof}
We  use ideas from \cite[Theorem 1.6]{HV}. 
Let $k$ be an algebraically closed field of characteristic zero and let $X$ be a smooth Fano variety of dimension $n\geq 5$ and index $n-3$ over $k$, as in Theorem \ref{HVthm}.
In particular, $-K_X=\mathcal O_X(n-3)$ and $h^0(X,\L)\geq n-2$. 
Let $D_1, \ldots, D_{n-2}$ be general elements of the linear system $|\L|$ on $X$ and let 
$$
Z_i:=X\cap D_1\cap \dots \cap D_i .
$$
By Theorem \ref{HVthm}, $Z_{n-3}$ is a Gorenstein threefold with canonical singularities and trivial canonical bundle.
The pair $(Z_{n-3},Z_{n-2})$ is therefore log canonical by a result of Kawamata \cite[Proposition 4.2]{kawamata}.
Applying repeatedly inversion of adjunction (see e.g.\ \cite[Theorem 7.5]{kollar}), the pair $(X,\sum_{i=1}^{n-2}D_i)$ is thus seen to be log canonical near $Z_{n-3}$.

For a contradiction, we assume that every element in $|\L|$ is singular.  
In particular, each $D_i$ is singular.
By Bertini's theorem, the singular locus of $D_i$ must be contained in the base locus of $\L$ and hence in the singular locus of $Z_{n-3}$.
Since $Z_{n-3}$ has only isolated singularities by Theorem \ref{HVthm}, we find that $D_i$ has only isolated singularities as well.
Let $x$ be such a singular point.
Since for all $i$, the divisor $D_i$ is a general element of $|\L|$, $x$ is a singular point of $D_i$ for all $i$.
Let $\tau:X'\to X$ be the blow-up of $X$ in $x$ with exceptional divisor $E$.
Since $X$ is smooth, $K_{X'}=\tau^\ast K_X+(n-1)E$.
Since $X$ is smooth and $D_i$ is singular at $x$, we see that $\tau^\ast D_i=D_i'+a_iE$ for some $a_i\geq 2$, where $D_i'$ denotes the proper transform of $D_i$.
We thus find
$$
K_{X'}+\sum_{i=1}^{n-2} D_i'=\tau^\ast (K_{X}+\sum_{i=1}^{n-2}D_i)+(n-1)E-\sum_{i=1}^{n-2}a_iE .
$$
Since $n\geq 5$ and $a_i\geq 2$ for all $i$, the discrepancy of $E$ with respect to the pair $(X,\sum_{i=1}^{n-2} D_i)$ is thus given by
$$
n-1-\sum_{i=1}^{n-2}a_i \leq n-1-2(n-2)=-n+3\leq -2 .
$$
This shows that $(X,\sum_{i=1}^{n-2} D_i)$ is not log canonical near $Z_{n-3}$, a contradiction.
This concludes the proof of the proposition. 
\end{proof}

With the above proposition, we can prove the main theorem of this section, which reduces Theorem \ref{quintictheorem} to a question about threefolds and fourfolds.

\begin{theorem} \label{thm:reduction}
Suppose Theorem \ref{quintictheorem} holds in dimension $3$ and $4$. 
Then it holds in any dimension $\ge 3$.
\end{theorem}
\begin{proof}
We prove the theorem by induction on the dimension.
Let $X$ be as in Theorem \ref{quintictheorem} and assume that $n=\dim X\geq 5$ and that we know Theorem \ref{quintictheorem} in lower dimensions.
By Kodaira vanishing, $\chi(X,\mathcal O_X(1))=h^0(X,\mathcal O_X(1))\geq n+2$.
Hence, Proposition \ref{prop:HV} implies that a general element $D\in |\L|$ is smooth.
If $\mathcal O_D(1)$ denotes the restriction of $\mathcal O_X(1)$ to $D$, then we have
$$
\mathcal O _D(1)^{n-1}=5,\ \ \chi(D,\mathcal O_D(1))=h^0(D,\mathcal O_D(1))\geq n+1\ \ \text{and}\ \ K_D=\mathcal O_D(4-n).
$$
Hence, $D$ satisfies the assumptions of Theorem \ref{quintictheorem} and so it is isomorphic to a quintic hypersurface by assumption.
But then by the Lefschetz hyperplane theorem, $\Pic D$ is generated by a line bundle whose self-intersection is 5 and so we find that $\mathcal O_D(1)$ must be the hyperplane bundle.
It thus follows from Theorem \ref{hyperplanesection} that the global sections of $\mathcal O_X(1)$ embed $X$ as a quintic hypersurface in $\CP^{n+1}_k$, as we want.
This concludes the proof of the theorem.
\end{proof}

\section{Quintic threefolds} 
In this section we aim to prove Theorem \ref{quintictheorem} in dimension three.
In fact, thanks to resolution of singularities in dimension three \cite[p.\ 1839]{CP}, which ensures the existence of the diagram (\ref{eq:diagr}), we are able to settle the case where the ground field has arbitrary characteristic.

\begin{theorem} \label{thm:n=3,d=5}
Let $k$ be an algebraically closed field (of arbitrary characteristic).
Let $X$ be a smooth projective threefold with $\Pic X/\sim_{num}=\ZZ [\L]$ for an ample line bundle $\L$ with 
$$
\L^3=5,\ \ h^0(X,\L)\geq 5\ \ \text{and}\ \ K_X=\mathcal O_X.
$$
Then $X$ is isomorphic to a quintic hypersurface in $\CP^{4}_k$.
\end{theorem}

To prove the above theorem, we consider the rational map $\phi: X\dashrightarrow Y\subset \PP^4$ given by five general sections of $\L$ and use the notation from Section \ref{subsec:notation}. 
By Lemmas \ref{lem:B,degY} and \ref{lem:dimY=1}, $\dim Y\in \{2,3\}$ and $\dim B\in \{0,1\}$.
We consider these cases in what follows.

\subsection{$\dim Y=2$ and $\dim B=0$} \label{subsec:dimY=2,dimB=0}

Let $C$ be a general fiber of $q$; so $C$ is a curve. Note that
$$
5=L^3=LM^2+LMF+L^2F=LM^2=(\deg Y)( \L\cdot p_*C)
$$
The surface $Y$ has degree $\ge 3$ by Lemma \ref{lem:B,degY}. 
So we must have that $\L\cdot p_*C =1$.

By Lemma \ref{lem:L^bM^n-b-1F}, $M^2F>0$ and so there is a component $F_0\subset F$ with $M^2F_0>0$.
We let $x\in X$ be the image of $F_0$ in $X$ ($F_0$ maps to $B$ and hence to a point in $X$ by Lemma \ref{lem:B,degY}). 
Note also that $x$ is a base point of $\L$.

Since $\L\cdot p_*C =1$, we see that for any general element $S\in |\L|$, there is a curve in $X$ which intersects $S$ with multiplicity one in $x$.
Hence $S$ is smooth at $x$. 
Let now $S_1,S_2\in |\L|$ be general elements.
Since $S_1$ is general, there is a general fibre $C_1$ of $q$ with $p_\ast C_1\subset S_1$.
Moreover, as we have seen above, $x$ is a smooth point of $S_1$.
Since $\L\cdot p_\ast C_1=1$, we then find that the curves $p_\ast C_1$ and $S_1\cap S_2$ meet with multiplicity one at the smooth point $x\in S_1$.
This shows that $T=S_1\cap S_2$ is a curve in $X$ with multiplicity one at $x$. 
However, since $\phi$ is given by $5$ general sections of $\L$, $T$ can be written as $p_\ast (q^\ast H\cap q^\ast H')$ for general hyperplanes $H,H'\in |\mathcal O_{\CP^4}(1)|$.
Since the component $F_0$ of $F$ lies above $x$ and dominates $Y$, we find that $T$ must have multiplicity at least $M^2 F_0\geq  \deg(Y)\geq 3$ at $x$, a contradiction.

\subsection{$\dim Y=2$ and $\dim B=1$}\label{sect321} 
With notation as in the previous case, we can write
$$
\L^2=(\deg Y) p_* C+B .
$$ 
Intersecting both sides with $\L$, we find that $3\le \deg Y\le 4$, because $\L\cdot B>0$, since $\dim B=1$ and $\L$ is ample.

If $\deg Y=3$, then $Y$ is a (possibly singular) cubic scroll.
Hence, $Y$ is covered by lines and so we conclude by Lemma \ref{lem:decomp}.

If $\deg Y=4$, we use the classification of surfaces of degree $4$ in $\PP^{4}$, see \cite{SD} or \cite[Theorem 8]{nagata}. This implies that $Y$ is either
\begin{enumerate}[(i)]
\item A projection of a quartic scroll in $\PP^5$ to $\PP^4$;
\item A projection of the Veronese surface to $\PP^4$;
\item An intersection of two quadrics; or 
\item A cone over a quartic curve in $\PP^3$.
\end{enumerate}

In cases (i) and (iv), $Y$ is a surface ruled by lines and we conclude by Lemma \ref{lem:decomp}. 

Similarly, in (ii), let $V\subset \PP^5$ be the Veronese surface and let $p: V\dashrightarrow Y$ be the projection. 
Recall that $V\simeq \PP^2$ embedded by the complete linear system $|\O_{\PP^2}(2)|$. 
Hence, there is a linear series $W\subset |\mathcal O_{\CP^2}(2)|$ of codimension one such that the induced map $\phi_W:\CP^2\to \CP^4$ has $Y$ as image.  
Consider the subspace $W'\subset |\mathcal O_{\CP^2}(2)|$ that is given by decomposable divisors, i.e.\ by sums of two lines.
Since $W$ is a linear subspace of codimension one and since each line on $\CP^2$ moves in a $2$-dimensional linear series, we conclude that $W\cap W'$ contains at least a one-dimensional family of decomposable divisors where both components move.
Since the hyperplane sections of $Y$ are exactly given by the elements of $W$, this contradicts  Lemma \ref{lem:decomp}, as we want.

In case (iii), we use the following well-known characterization of quadrics over algebraically closed fields of arbitrary characteristic.

\begin{lemma} \label{lem:quadric}
Let $k$ be an algebraically closed field and let $Q\subset \CP^{n+1}_k$ be a quadric hypersurface.
Assume that $X$ is integral.
Then $Q$ is the cone over a smooth quadric $Q_0$. 
Moreover, if $Q$ is smooth, then there are homogeneous coordinates $x_0,\dots ,x_{n+1}$, such that $Q$ is given by $\sum_{i=0}^{m-1} x_ix_{i+m}=0$ if $n+2=2m$ for some integer $m$, or $\sum_{i=0}^{m-1} x_ix_{i+m}+x_{n+1}^2=0$ if $n+2=2m+1$ for some integer $m$. 
\end{lemma}

In the setting of the above lemma, we will say as usual that $Q$ has rank $r$ if $Q$ is a cone over a smooth quadric of dimension $r-2$.

To settle case (iii), let now $Y$ be the intersection of two quadrics $Q_1,Q_2$.
We note first that the pencil spanned by $Q_1$ and $Q_2$ contains a singular quadric $Q_0$. 
In particular, $Q_0$ contains a $1$-dimensional family of planes $P_t$. 
Each plane $P_t$ must intersect $Y$ in a conic curve $C_t$, and each $C_t$ lies in a pencil of hyperplane sections $H_t$ so that $H_t=C_t+C_t'$ for a residual conic curve $C_t'$. 
If the residual conic $C_t'$ moves on $Y$, then we are done via Lemma \ref{lem:decomp}.
Otherwise, $C'_t$ must be the base locus of the family of planes $P_t\subset Q_0$ and so $Q_0$ has rank three and $C'_t$ is given by the singular line $\ell$ of $Q_0$ (counted with multiplicity two).  
Since $Q_0$ has rank three, intersecting it with a general hyperplane $H$ containing $\ell$ gives a quadric of rank two. In particular, $H|_Y$ decomposes into a union of $\ell$ (with multiplicity two) and two lines. Hence $Y$ is ruled by lines, and we are again done by Lemma \ref{lem:decomp}.

\subsection{$\dim Y=3$} \label{subsec:dimY=3}
When $Y$ is a threefold in $\PP^4$, we note that 
\begin{equation}\label{ineqs3}
5=L^3 = L^2 M \ge L M^2\ge M^3=(\deg q)(\deg Y) .
\end{equation}
Once again, $\deg Y\ge 2$ by Lemma \ref{lem:B,degY}. 
Also we may assume $\deg q\ge 2$ by Proposition \ref{prop:degq=1}. 
This leaves the case $(\deg Y,\deg q)=(2,2)$. 
In particular, $M^3=4$ and so Lemma \ref{lem:Hodgeindex} shows that the base locus $B$ consists of a single point.\footnote{
If $\operatorname{char} k=0$, a general section of $\L$ is smooth by Lemma \ref{lem:Hodgeindex}, but we cannot use this to conclude inductively, because because of the Horikawa examples in dimension 2 \cite{horikawa}. 
}

Since $\deg Y=2$, $Y$ is a quadric. 
Since $Y$ is integral, it must either be smooth, or the cone over a smooth conic, or the cone over a smooth quadric surface, cf.\  Lemma \ref{lem:quadric}. 

If $Y$ is the cone over a smooth conic in $\PP^2$, then it is ruled by $\PP^2$'s and we conclude by Lemma \ref{lem:decomp}.

If $Y$ is the cone over a smooth quadric surface $Q\subset \CP^3$, then the union of two lines in different rulings of $Q$ yields a hyperplane section of $Q$ with two movable components.
Taking the cone over this divisor, we find a hyperplane section of $Y$ which has two movable components and so we get a contradiction using Lemma \ref{lem:decomp}. 

If $Y$ is a smooth quadric, the class $H^2$ is divisible by two in $\CH^2(Y)$.
Since $\dim(B)=0$, we have $\L^2=p_\ast q^* H^2$ on $X$, and so we conclude that this class is divisible by two as well, which contradicts the fact that $\L^3=5$.

This concludes the proof of Theorem \ref{thm:n=3,d=5}.

\section{Quintic fourfolds}
We now prove Theorem \ref{quintictheorem} for fourfolds (over an algebraically closed field $k$ of characteristic zero). 
The strategy of proof follows the argument for the three-dimensional case, but some additional difficulties arise in dimension four.  
We use the notation of Section \ref{subsec:notation} and consider the diagram
$$\begin{tikzcd}
 & W \arrow[ld, "p"'] \arrow[rd, "q"] &  \\
X \arrow[rr, "\phi", dashed] &  & Y \subset \PP^5
\end{tikzcd}$$
By Lemmas \ref{lem:B,degY}, \ref{lem:dimY=1} and \ref{lem:Hodgeindex}, we have  
$$
\dim Y\in \{2,3,4\},\ \  6-\dim Y \leq \deg Y\leq 5\ \ \text{and}\ \  \dim B\in \{0,1,2\} .
$$
We deal with these cases in what follows.

\subsection{$\dim Y=2$ and $\deg Y=4$}

In this case $Y$ is a surface of minimal degree in $\CP^5$ and so, using\cite{EH}, we see that it is either a cone over a rational normal curve in $\CP^4$, the Veronese surface, $\CP^1\times \CP^1$ (embedded via $\mathcal O(1,1)$), or a quartic scroll.
In each case, the hyperplane divisor $H$ can be written as the sum of two movable divisors on $Y$, which once again contradicts Lemma \ref{lem:decomp}.

\subsection{$\dim Y=2$ and $\deg Y=5$}\label{dimY2degY5}

Let $S$ be a general fibre of $q$.
Since $\dim B\leq 2$ by Lemma \ref{lem:B,degY}, 
$$
\L^2=\deg(Y)p_\ast S+R
$$
for an effective codimension two class $R$. 
Since $\deg(Y)=5$, $\L^4=5$ and $\L$ is ample, we conclude $R=0$.
Hence, $\L^2=5p_\ast S$ and so $5=\L^4=25(p_\ast S)^2$, which is a contradiction.

\subsection{$\dim Y=3$ and $\deg Y=3$}

In this case $Y$ is a threefold of minimal degree in $\CP^5$.
If it is smooth, then it is the Segre cubic threefold $\CP^1\times \CP^2$.
If $Y$ is singular, then it is a cone over a surface of minimal degree in $\CP^4$, which is ruled by lines, cf.\ Section \ref{sect321}.
In either case $Y$ is covered by 2-planes and we get a contradiction from Lemma \ref{lem:decomp}.

\subsection{$\dim Y=3$ and $\dim B=0$} \label{subsec:dimY=3,dimB=0}

Let $C$ be a general fibre of $q$ (so $C$ is a smooth curve). 
Since the base locus $B$ is assumed to be zero-dimensional, we have $\L^3=\deg(Y)p_\ast C$ and so $\deg(Y)=5$ and $\L \cdot p_\ast C=1$.
The same argument as in the case of threefolds now applies, see Section \ref{subsec:dimY=2,dimB=0}.

\subsection{$\dim Y =3$ and $\deg Y=5$}

Let $C$ be a general fibre of $q$.
Since $\dim B\leq 2$ by Lemma \ref{lem:B,degY}, we have $L^3F=0$ and so
$$
5=L^4=LM^3+LM^2F+L^2MF.
$$
Here we have $LM^3=\deg(Y)(\L \cdot p_\ast C)\geq 5$, because $\L$ is ample, $p_\ast C\neq 0$ (as $C$ is a general fibre of $q$) and $\deg Y= 5$ by our assumptions.
Since the remaining terms are non-negative by Lemma \ref{lem:L^bM^n-b-1F}, we find $LM^2F=0$ and $L^2MF=0$.
By Lemmas \ref{lem:B,degY} and \ref{lem:L^bM^n-b-1F}, this implies $\dim B=0$ and so we conclude via the case treated in Section \ref{subsec:dimY=3,dimB=0}.

\subsection{$\dim Y=3$, $\deg Y=4$ and $\dim B=2$}

As in the previous case, we find
$$
5=L^4=LM^3+LM^2F + L^2MF .
$$
Moreover, $LM^3=\deg Y\cdot (\L \cdot p_\ast C)$, where $C$ is a general fibre of $q$.
Since $\deg Y=4$, we thus get
$$
LM^3=4 .
$$
Hence, $LM^2F + L^2MF =1$ and so Lemma \ref{lem:L^bM^n-b-1F} shows
$$
LM^2F=0\ \ \text{and}\ \ L^2MF=1 ,
$$
because $\dim B=2$.

Let now $S$ be the intersection of a general element of $|M|$  with a general element of $|rL|$ for some $r\gg0$.
Then $S$ is a smooth surface on $W$.
We let $f$, $m$ and $l$ denote the restrictions of $F$, $M$ and $L$ to $S$.
Then we have
$$
m^2=rLM^3=4r,\ \ lf=rL^2MF=r \ \ \text{and}\ \ mf=rLM^2F=0.
$$
Since $m$ is a nef divisor on $S$, $m^2=4r$ shows that it is big and nef.
Moreover, $f$ is an effective divisor on $S$ and we have $mf=0$ and $f^2=(l-m)f=lf=r>0$, which contradicts  the Hodge index theorem on $S$.
This concludes the present case. 

\subsection{$\dim Y=3$, $\deg Y=4$ and $\dim B=1$} \label{subsec:dimY=3,degY=4,dimB=1}

Since $Y$ is a threefold, $M^4=0$.
Moreover, $L^2F=0$ because $\dim B=1$.
We thus find $5=L^4=FM^3+LM^2F$. 
Here, $FM^3=LM^3=\deg Y\cdot (\L \cdot p_\ast C)$, where $C$ is a general fibre of $q$.
Hence, 
$$
FM^3=4,\ \ \L \cdot p_\ast C=1\ \ \text{and} \ \ LM^2F=1 .
$$ 
Since $M^3=4C$, we have $\L^3=4 p_*C+B$.
Since $\L^4=5$, we conclude that $\L\cdot B=1$ and so $B$ is given by an integral curve on $X$.

We write $F=F'+F''$, where $F'$ denotes the union of all components that map to points on $X$, while $F''$ denotes the union of all components that dominate $B$.  
Then, $LM^2F=LM^2F''=1$. 
Take a smooth divisor $D\in |\O_X(m)|$ for $m\gg 0$ so that the intersection $D\cap B$ is transversal, i.e. consists of $m$ distinct points. 
Then on $W$, the 2-cycle $p^\ast D\cdot F$ is represented by $m$ fibers of $F''\to B$. 
In particular, this shows that the class $LF$ is represented by an effective cycle, namely the fiber over a general point of $B$. 
Since $LM^2F=1$, we find that this cycle is mapped to a plane in $Y\subset \CP^5$; that is, there is exactly one component of $LF$ which is not contracted via $q_\ast$ and this component is mapped to a plane in $Y$. 
Since $Y$ is a threefold, this plane cannot move on $Y$ (by Lemma \ref{lem:decomp}). 
It follows that none of the components of $F''$ dominate $Y$. 
Hence, if $F_0\subset F$ is a component which dominates $Y$, then $F_0\subset F'$, i.e. $F_0$ maps to a point in $X$.
Since $M^3=\deg(Y)C=4C$ on $W$, the equality $FM^3=4$ shows that there is a unique such component $F_0$ and we denote its image in $X$ by $x\in X$.

Let $D\in |\L|$ be a general element. 
Since $\L\cdot p_\ast C=1$, $D$ must be smooth at $x$.
Let $S\in |\L|_D|$ be the intersection of $D$ with another general element of $|\L|$.
Then, as in Section \ref{subsec:dimY=2,dimB=0}, the fact that $\L \cdot p_\ast C=1$ implies that $S$ is smooth at $x$.
Finally, let $K\in |\L|_S|$ be the intersection of $S$ with a third general element of $|\L|$.
Then, as before, $\L\cdot p_\ast C=1$ implies that $K$ is smooth at $x$.
On the other hand, $K=B+4p_\ast C$ has multiplicity at least $4$ at $x$ and so it is not smooth at $x$.
This contradiction establishes the case $\dim Y=3$, $\deg Y=4$ and $\dim B=1$, as we want.

\subsection{$\dim Y=4$} \label{subsec:dimY=4}

By Theorem \ref{HVthm}, a general element $D$ of $|\L|$ is a canonical threefold with trivial canonical bundle.
Let $D'\in |M|$ be the proper transform of $D$.
Then $D'$ is a resolution of $D$ and so $h^0(D',K_{D'})=h^0(D,K_D)=1$, because $D$ has canonical singularities. 
Applying \cite[Proposition 2.1]{Kob} to $D'$, we find that 
$$
(M|_{D'})^3\geq 2h^0(D',M|_{D'})-6.
$$
Since $h^0(W,M)=h^0(X,\L)\geq 6$ by assumption, we have $h^0(D',M|_{D'})\geq 5$ and so the above inequality shows
$M^4\geq 4$. 
Hence, the generic element $D\in |\L|$ is smooth by Lemma \ref{lem:Hodgeindex}. 
Finally, by the three-dimensional case, $D$ is a quintic threefold, and so we conclude by Proposition \ref{hyperplanesection}.

This completes the proof of Theorem \ref{quintictheorem} in the case of fourfolds.
The general statement follows therefore from Theorem \ref{thm:reduction} and the three-dimensional case proven in Theorem \ref{thm:n=3,d=5} above.

\subsection{Proof of Corollary \ref{corprojdef}}

\begin{proof}[Proof of Corollary \ref{corprojdef}]
Let $X$ be a projective manifold of dimension $n\geq 3$ which deforms to a quintic hypersurface $Y\subset \CP^{n+1}_{\CC}$ via a sequence of K\"ahler deformations.
Then $b_2(X)=b_2(Y)=1$ and so any fibre in the above sequence of deformations is a projective manifold.
It follows easily that the ample generator $\mathcal O_X(1)$ of $\Pic X$ must deform to the ample generator $\mathcal O_Y(1)$ of $\Pic Y$.
In particular,
$$
\mathcal O_X(1)^n=\mathcal O_Y(1)^n=5.
$$
Moreover, since Euler characteristics are constant in flat families,
$$
\chi(X,\mathcal O_X(1))=\chi(Y,\mathcal O_Y(1))=n+2.
$$
Finally, $K_X$ deforms to $K_Y$ and so $K_X=\mathcal O_X(-n+2)$.
Hence, Theorem \ref{quintictheorem} applies and we find that $X$ is isomorphic to a quintic hypersurface.
This proves the corollary.
\end{proof}

\begin{remark}
Corollary \ref{corprojdef} says that a complex projective manifold $X$ which deforms to a quintic hypersurface $Y$ via a sequence of K\"ahler deformations must itself be a quintic hypersurface. 
We do not know if the corresponding statement holds true if $X$ and $Y$ are only deformation equivalent as complex manifolds (and $n$ is even). 
Here one would have to rule out the situation where the ample generator of $\Pic X$ deforms to the anti-ample generator of $\Pic Y$. 
Since the top self-intersection of an ample line bundle is positive, this could only happen in the case where $n$ is even. 
\end{remark}

\section{Septic threefolds} \label{sec:n=3,d=7}

In this section, we prove the following theorem.

\begin{theorem} \label{thm:n=3,d=7}
Let $k$ be an algebraically closed field (of arbitrary characteristic) and let $X$ be a smooth  projective threefold over $k$ with $\Pic X/\sim_{num} =\ZZ [\L]$ for an ample line bundle $\L$ with
$$
\L^3=7,\ \ h^0(X,\L)\geq 5\ \ \text{and}\ \ K_X=\mathcal O_X(2) .
$$
Then $X$ is isomorphic to a septic threefold in $\CP^4_k$.
\end{theorem}

As before, we use resolution of singularities in dimension three over arbitrary algebraically closed fields \cite[p.\ 1839]{CP}, which ensures the existence of  the diagram
$$\begin{tikzcd}
 & W \arrow[ld, "p"'] \arrow[rd, "q"] &  \\
X \arrow[rr, "\phi", dashed] &  & Y \subset \PP^{4}
\end{tikzcd}$$
where $\phi$ is the rational map given by five general sections of $\L$, and where $L=p^\ast \L=M+F$ for a base point free divisor $M=q^\ast H$ and a divisorial fixed part $F$ of $p^\ast \L$.
By Lemma \ref{lem:dimY=1}, $\dim Y\in \{2,3\}$ and we will deal with these cases separately in the following subsections.

\subsection{$\dim Y=2 $} 
If $\dim B=0$, then the proof of Section \ref{subsec:dimY=2,dimB=0} carries over verbatim, and we get a contradiction.
Hence, $\dim B=1$ by Lemma \ref{lem:B,degY}.

Since $Y\subset \PP^4$ is non-degenerate, it has degree at least $3$ by Lemma \ref{lem:B,degY}.
If $\deg Y=3$, then, as before, $Y$ is a surface of minimal degree and so it is covered by lines, which is impossible by Lemma \ref{lem:decomp}. 
If $\deg Y=4$, then the argument of Section \ref{sect321} applies, and shows that this is impossible. 
We may therefore assume $\deg Y\geq 5$.
We then have
$$
7=L^3=LM^2+LMF
$$
where $LM^2=\deg Y\cdot (\L\cdot p_\ast C)$, where $C$ denotes a general fibre of $q$.
Since $\deg Y\geq 5$, we conclude that $\L\cdot p_\ast C =1$ and $LMF=7-\deg Y \leq 2$.

If $\deg Y=7$, then $LMF=0$ and so Lemma \ref{lem:L^bM^n-b-1F} implies $\dim B=0$, which contradicts our earlier observation that $B$ must be one-dimensional. 
Hence, $\deg Y<7$ and so $LMF>0$.
Since $LMF\leq 2$, the effective curve class $LF$ on $W$ maps to a curve $K\subset Y$ which is either a line or a conic in $\PP^4$. 
In any case, there is a 1-dimensional family of hyperplanes that contain $K$ and hence the residual divisor $K'=H-K$ on $Y$ has a component which moves in a family of dimension at least one. 
Hence, if $K$ moves on the surface $Y$, we get a decomposition of $H$ satisfying Lemma \ref{lem:decomp} and we are done.

Therefore, we can assume that the curve class $q_\ast LF$ does not move and we conclude that any component $F_0$ of $F$ which dominates $Y$ must map to a point on $X$.
Since $\L\cdot p_\ast C=1$, and $FM^2=LM^2=\deg Y$ (as $M^3=0$), there is exactly one such component $F_0$ and we denote its image in $X$ by $x$.
It follows that $\L^2=\deg Y\cdot p_\ast C+B$ has multiplicity at least $\deg Y$ at $x$ and so we get a contradiction as in Sections \ref{subsec:dimY=2,dimB=0} and \ref{subsec:dimY=3,degY=4,dimB=1}. 
Indeed, a general element $S\in |\L|$ will be smooth at $x$, because $S\cdot p_\ast C=1$.
Since $S$ is general, there is another general fibre $C'$ of $q$ such that $p_\ast C'\subset S$ and so $\L\cdot p_\ast C'=1$ implies that a general element of $|\L|$ restricted to $S$ is smooth at $x$; that is, the intersection of two general elements of $|\L|$ must be smooth at $x$, which is a contradiction to the above observation that any such intersection has multiplicity at least $\deg(Y)\geq 2$ at $x$.
This finishes the proof in the case $\dim(Y)=2$.

\subsection{$\dim Y=3$} \label{subsec:septic:dimY=3}
Since $\dim Y=3$, $M^3=\deg q \cdot \deg Y$.
By Lemmas \ref{lem:B,degY} and \ref{lem:Hodgeindex}, we thus get $2\leq M^3\leq 7$.
We will deal with these cases in what follows.

\subsubsection{$M^3\in \{2,3,5,7\}$}
If $M^3=7$, then we conclude via Lemmas \ref{lem:Hodgeindex} and \ref{lem:H^n=d}. 
If $M^3=2,3$ or $5$, then $\deg q=1$ because $M^3=\deg q\cdot \deg Y$ and $\deg Y\geq 2$ (as it is not contained in a hyperplane).
Hence, $q$ is birational and we conclude via Proposition \ref{prop:degq=1}.

\subsubsection{$M^3=6$}
Since $M^3=6$, $\L$ has zero-dimensional base locus by Lemma \ref{lem:Hodgeindex}.
Moreover, $6=M^3=\deg q\cdot \deg Y$ with $\deg Y\geq 2$ and $\deg q\geq 2$ by Proposition \ref{prop:degq=1}. 
Hence $(\deg q,\deg Y)$ is either $(2,3)$ or $(3,2)$.
If $Y$ is a quadric, then (using that $\dim B=0$) we can conclude as in Section \ref{subsec:dimY=3}.
We are therefore left with the first case, that is, $Y$ is a cubic threefold and $\deg(q)=2$.

By Lemma \ref{lem:Hodgeindex}, $\L$ has a single base point with multiplicity one and so we may assume that $W$ is the blow-up of $X$ in a single point with exceptional divisor $F\cong \CP^2$.
In particular, 
$$
7=L^3=M^3+M^2F
$$ 
and so $M^2F=1$, which implies that $F$ maps (isomorphically) onto a plane $P$ in $Y$.
Let $H=P+S$ be a general hyperplane section of $Y$ which contains $P$.
Then $S$ is a quadric surface whose generic point is a smooth point of $Y$.
We can find a pencil of quadric hypersurfaces in $\CP^4$ that contain $S$. 
Let $Q$ be the restriction of a general such quadric to $Y$. 
Then $Q=S+S'$ for a degree four surface $S'$ whose generic point is also a smooth point of $Y$. 
On $W$, we get $2M=q^\ast Q=q^*S+q^*S'$, where $q^\ast S$ denotes the union of components that map to $S$ and $q^\ast S'$ denotes the remaining components. 

Since $Q$ does not contain the plane $P$, $q^\ast S$ and $q^\ast S'$ do not contain $F$ and so no component of $q^\ast S$ or $q^\ast S'$ is contracted by $p_\ast$.
Since $\Pic X$ is generated by $\L$, the divisors $q^\ast S$ and $q^\ast S'$ on $W$ need therefore both map down to elements of $|\L|$ on $X$.
(In particular, since all elements of $|\L|$ are irreducible, $q^\ast S$ and $q^\ast S'$ must be irreducible as well, which justifies our notation.)
Hence, we must have linear equivalences $q^*S=M+bF$ and $q^*S'=M+cF$ for some integers $b$ and $c$.
Since $S$ and $S'$ lie generically in the smooth locus of $Y$ and since a smooth point on $Y$ pulls back to a zero cycle of degree two on $W$, we compute the intersection with $M^2$ as follows:
$$
q^*S\cdot M^2=2\deg(S)=4\ \ \text{and}\ \ q^*S'\cdot M^2=2\deg(S')=8 .
$$
Since $M^3=6$ and $M^2F=1$, we conclude from this that $b=-2$ and $c=2$. 
But then let $\ell\subset F$ be a general line in $F\cong \CP^2$. 
We have  $F\cdot \ell=-1$.
Moreover, $\ell=-F^2$ and so we find
$$
M\cdot \ell=-MF^2=-MLF+M^2F=1 .
$$ 
This implies that 
$$
q^*S'\cdot \ell= (M+2F)\cdot \ell=-1 .
$$ 
That is, $q^*S'$ contains $\ell$, which is absurd, as it implies that $q^\ast S'$ contains $F$ and so $S'$ would need to contain $P$.  
This concludes the case $M^3=6$.

\subsubsection{$M^3=4$}
As $M^3=\deg Y\cdot \deg q$ and $\deg q=1$ is impossible by Proposition \ref{prop:degq=1}, we reduce to the case when $\deg q=\deg Y=2$.
By the argument in Section  \ref{subsec:dimY=3} we may further assume that $Y$ is a smooth quadric threefold.
Hence, the class $H^2$ is divisible by two in $\CH^2(Y)$.
If $\dim(B)=0$, we have $\L^2=p_\ast q^* H^2$ on $X$, and so we conclude that this class is divisible by two as well, which contradicts the condition that $\L^3=7$.
Hence we must have $\dim B>0$ and so
Lemma \ref{lem:B,degY} implies $\dim B=1$.

We have $$7=L^3=LFM+FM^2+M^3.$$
Since $M^3=4$, we find $LFM+FM^2=3$.
Also, as $\dim B=1$, we have $LFM>0$ by Lemma \ref{lem:L^bM^n-b-1F}. 

Let us now show that $LFM=1$ and $FM^2=2$. 
To this end, note that $M^2=2q^\ast \ell$, where $\ell$ denotes the class of a line on the smooth quadric threefold $Y$. In particular, this implies that $FM^2$ is even.
If $FM^2=0$, then let $T\in |M|$ be a general element and let $l,f,m$ be the restrictions of $L,F,M$ to $T$.
Then $f$ is an effective curve class with $mf=0$ and $f^2=lf>0$. 
This contradicts the Hodge index theorem \cite[Theorem 10.1]{KK}, 
 applied to a resolution of singularities of $T$ (which exists as $T$ is a surface), because $f$ is effective and $M|_T$ is base point free and big.
Hence, $FM^2\geq 2$.
As  $LFM+FM^2=3$, this proves the claim.

Let $\ell$ denote a general line on $Y$.
Then $M^2=2q^\ast \ell$.
Since $FM^2=2$, we find that $F \cdot q^\ast \ell =1$.  
Therefore, $F$ contains exactly one component which is not contracted via $q$ and we denote that component by $F_0\subset F$. 
Since $Y$ is a smooth quadric threefold, $FM^2=F_0M^2=2$ implies that $q_\ast F_0$ is a hyperplane section of $Y$, and the morphism $q|_{F_0}: F_0\to Y$ is generically injective.

Let $H_1$ and $H_2$ denote general hyperplane sections of $Y$.
Since $M^2F_0=2$, the curve $C:=q^\ast (H_1H_2)$ meets $F_0$ at two general points $w$ and $w'$ (via $q$ these points are mapped to the intersection of the quadric surface $q_\ast F_0$ with $H_1$ and $H_2$) 
and we let $x:=p(w)\in X$. 
The image of $w$ in $Y$ is a general point of $q(F_0)$.
The divisor $D_i:=p_\ast q^\ast H_i$ is a general element of $|\L|$ and we have
$$
D_1D_2=p_\ast C+B .
$$ 

By construction, $x$ lies on the above 1-cycle and we claim that it is in fact a singular point of $p_\ast C+B$.
Indeed, if $F_0$ maps to a curve in $X$, then it maps to a component of $B$ and so $x$ lies on $B$ and on $p_\ast C$, which proves the claim.
Otherwise, $F_0$ maps to a point on $X$ and so we conclude that $p_\ast C$ must have multiplicity two at $x$, because $F_0 \cdot M^2=2$.
We have thus seen that $x$ is a singular point of the intersection $D_1D_2$ of two general elements of $|\L|$.

We will now show that this is a contradiction.
To this end, recall that $x=p(w)$, where $w$ is a general point on $F_0\subset W$. 
Let $\ell$ be a general line on $Y$ which passes through $q(w)$ (since $q_\ast F_0$ is a hyperplane section of $Y$, and $q(w)\in q_\ast F_0$ is a general point, $\ell$ is in fact a general line on $Y$).

Let $H\in |\mathcal O_Y(1)|$ be a general hyperplane section.
Then $S:=p_\ast q^\ast H$ is a general element of $|\L|$ and so $p^\ast S=q^\ast H+F\in |L|$.
Using the projection formula, we  get
$$
S\cdot p_\ast q^\ast \ell = p_\ast(p^\ast S\cdot q^\ast \ell) =p_\ast((q^\ast H+F)\cdot q^\ast \ell) .
$$
Using this, we conclude further
$$
S\cdot p_\ast q^\ast \ell = p_\ast(q^\ast (H\cdot \ell)) + p_\ast(Fq^\ast \ell)=p_\ast(q^\ast (H\cdot \ell)) + p_\ast w ,
$$ 
because $\ell$ passes through $q(w)$ and so $F q^\ast \ell=w$, where we note that the latter intersection consists of a single reduced point because $F q^\ast \ell=\frac{1}{2}FM^2=1$.
Here, $p_\ast(q^\ast (H\cdot \ell))$ is an effective zero cycle of degree two which has no base point if we move the hyperplane $H$; since $H$ is general, this zero cycle does therefore not contain $x=p(w)$.
That is, $S\cdot p_\ast q^\ast \ell$ has multiplicity one at $x$ and so $S$ must be smooth at $x$. 

Since $Y$ is a smooth quadric, the general hyperplane section $H\in |\mathcal O_Y(1)|$ from above is a smooth quadric surface and so it is covered by lines.
In particular, we find a line $\ell'\subset H$ through the point $q(w)$  
and we note that the $1$-cycle $p_\ast q^\ast \ell'$ lies on $S$.

Let $H'\in |\mathcal O_Y(1)|$ be another general hyperplane section and let $D:=p_\ast q^\ast H'\in |\mathcal O_X(1)|$.
Since $H'$ is general, $p^\ast D=q^\ast H'+F$ and restricting this equality to the surface $\widetilde S=q^\ast H\subset W$ (which is the proper transform of $S$), we find 
$$
p|_{\widetilde S}^\ast (D|_S)=q|_{\widetilde S}^\ast (H'|_H)+F|_{\widetilde S}.
$$
Let us denote by $T=D|_S$ the restriction of $D$ to $S$.
Intersecting this (Cartier) divisor with the $1$-cycle $p_\ast q^\ast \ell'$  on $S$, we thus get
$$
T\cdot p_\ast q^\ast \ell'=p_\ast(p|_{\widetilde S}^\ast (D|_S) \cdot q^\ast \ell')=p_\ast (q|_{\widetilde S}^\ast (H'|_H)\cdot q^\ast \ell'+F|_{\widetilde S}\cdot q^\ast \ell')=p_\ast (q^\ast (H'\cdot \ell'))+p_\ast (F\cdot q^\ast \ell') ,
$$
where in the last equality we used that $q^\ast \ell'$ is a $1$-cycle on $W$ that lies on $\widetilde S$. 
Moreover, since $\ell'$ passes through $q(w)$, find that
$$
p_\ast (F\cdot q^\ast \ell')=p_\ast w=x
$$
Since $H'$ is general, $p_\ast (q^\ast (H'\cdot \ell'))$ is a zero cycle that is supported on general points of $p_\ast q^\ast \ell'$ and so the above computation shows that $T\cdot p_\ast q^\ast \ell'$ contains $x$ with multiplicity one.
That is, $T$ is smooth at $x$.
On the other hand, $T=p_\ast C+B$ by construction and this curve is singular at $x$, as we have shown above.
This contradiction concludes the proof of Theorem \ref{thm:n=3,d=7}.

\section{Proof of Theorem \ref{thm:question1}} \label{sec:question1}  

To apply our results to Question \ref{question}, we will use the following key lemma.

\begin{lemma}\label{lem:question1:reduction}
Let $X$ be a smooth projective variety of dimension $n\geq 3$ over a field $k$.
Assume that $X$ is a smooth specialization of a hypersurface of prime degree $d$ in $\CP^{n+1}$. 
We denote by $X_{\overline k}$ the base change of $X$ to the algebraic closure of $k$.
Then the following holds:
\begin{enumerate} [(a)]
\item 
$X_{\overline k}$ is the smooth specialization of a hypersurface of degree $d$ in $\CP^{n+1}$.
\label{item:lem:question1:reduction:1}
\item 
If $X_{\overline k}$ is isomorphic to a hypersurface of degree $d\geq 2$, then $X$ is isomorphic to a hypersurface of degree $d$. 
\label{item:lem:question1:reduction:2} 
\item  \label{item:lem:question1:reduction:3}
Assume $k=\overline k$.
Then $\Pic X /\sim_{num}=\ZZ [\L]$ for an ample line bundle $\L$ with
$$
\L^n=d,\ \ h^0(X,\L)\geq n+2\ \ \text{and}\ \ K_X=\mathcal O_X(-n-2+d) .
$$
\end{enumerate}
\end{lemma}
\begin{proof}
Let $A$ be a local ring with residue field $\kappa$ and let $\kappa'$ be a field extension of $\kappa$.
By inflation of local rings, there is a local $A$-algebra $B$ with residue field $\kappa'$, see \cite[Appendice, Corollaire du Th\'eor\`eme 1]{bourbaki}.
This immediately implies item (\ref{item:lem:question1:reduction:1}) of the lemma. 

To prove item (\ref{item:lem:question1:reduction:2}), note first that we can choose a sequence of smooth specializations over discrete valuation rings to pass from a smooth hypersurface of degree $d$ to $X$.  
Applying the specialization homomorphism on Picard groups repeatedly, we find that $X$ admits a line bundle $\mathcal O_X(1)$ with $\mathcal O_X(1)^n=d$. 
Assuming that $X_{\overline k}$ is isomorphic to a hypersurface of degree $d$, the Lefschetz hyperplane theorem (see  \cite[Expos\'e XII, Corollaire 3.7]{SGA2}) implies that its Picard group is infinite cyclic and generated by an ample line bundle $\mathcal O_{X_{\overline k}}(1)$ with $\mathcal O_{X_{\overline k}}(1)^n=d$.
Hence, $\mathcal O_{X_{\overline k}}(1)$ must be the base change of $\mathcal O_X(1)$.

By upper semicontinuity, $\mathcal O_X(1)$ has at least $n+2$ sections.
Since $\mathcal O_{X_{\overline k}}(1)$ has exactly $n+2$ sections (as $d\geq 2$), we find a basis of $H^0(X_{\overline k},\mathcal O_{X_{\overline k}}(1))$ which is defined over $k$.
In particular, the embedding $X_{\overline k}\hookrightarrow \CP^{n+1}_{\overline k}$ that is induced by the global sections of $\mathcal O_{X_{\overline k}}(1)$ can be defined over $k$ and so $X$ is isomorphic to a hypersurface of degree $d$ in $\CP^{n+1}_{k}$, as claimed.

To prove item (\ref{item:lem:question1:reduction:3}), 
we use once again the Lefschetz hyperplane theorem, asserting that  
the Picard group of a smooth projective hypersurface (over any field) of dimension $n\geq 3$ is freely generated by the restriction of $\mathcal O(1)$, see \cite[Expos\'e XII, Corollaire 3.7]{SGA2}. 
Since $X$ is the smooth specialization of a smooth hypersurface, we get as in the proof of part (\ref{item:lem:question1:reduction:2}) an ample line bundle $\mathcal O_X(1)\in \Pic X$ as specialization of $\mathcal O(1)$.
In particular, $\mathcal O_X(1)^n=d$, $h^0(X,\mathcal O_X(1))\geq n+2$ by upper semicontinuity and $K_X=\mathcal O_X(-n-2+d)$.
Next recall that $\Pic X/\sim_{num}$ is a finitely generated group (see \cite[p.\ 145, Th\'eor\`eme 2]{Neron}) which is torsion free and so it is a finitely generated free group.
Using the Kummer sequence, we see that the rank of this group is bounded from above by the rank of $H^2_{\text{\'et}}(X,\ZZ_\ell)$ for all primes $\ell\neq p$.  
By the smooth base change theorem for \'etale cohomology \cite[VI.4.2]{milne}, $H^2_{\text{\'et}}(X,\ZZ_\ell)=\ZZ_\ell$ and so we find that $\Pic X/\sim_{num}\cong \ZZ$. 
Since $\mathcal O_X(1)^n=d$ is a prime number, we see that even up to numerical equivalence, $\mathcal O_X(1)$ can not be a non-trivial multiple of another line bundle and so $\Pic X/\sim_{num}=[\L]\ZZ$, as we want.
This concludes the lemma.
\end{proof}

\begin{proof}[Proof of Theorem \ref{thm:question1}]
Let $X$ be a smooth projective variety of dimension $n$ over $k$ which is the smooth specialization of a hypersurface of degree $d$, satisfying assumption (\ref{item:cor:question1:0}), (\ref{item:cor:question1:1}) or (\ref{item:cor:question1:2}) of Theorem \ref{thm:question1}.
By item (\ref{item:lem:question1:reduction:1}) of Lemma \ref{lem:question1:reduction}, $X_{\overline k}$ is the specialization of a hypersurface of degree $d$ and so it satisfies item (\ref{item:lem:question1:reduction:3}) of Lemma \ref{lem:question1:reduction}.
By Theorems \ref{quintictheorem}, \ref{thm:n=3,d=7} and \ref{thm:n=3,d=5}, it follows that $X_{\overline k}$ is isomorphic to a hypersurface of degree $d$.
It thus follows from item (\ref{item:lem:question1:reduction:2}) of Lemma \ref{lem:question1:reduction} that $X$ is isomorphic to a hypersurface of degree $d$ as well.
This concludes the theorem. 
\end{proof}

\section{The families of Horikawa, Griffin and Reid}\label{sec:reid}
It is interesting to compare the proof of Theorem \ref{quintictheorem} with Horikawa's families of quintic surfaces. 
As mentioned in the introduction, Horikawa proved that the moduli space of minimal complex algebraic surfaces with invariants $p_g=4$, $q=0$ and $K_X^2=5$ has two 40-dimensional components $I$ and $II$. 
The component $I$ correspond to quintic surfaces, while for $II$ the canonical system $|K_X|$ has a single base point, and the rational map $\phi_{|K_X|}:X\dashrightarrow \PP^3$ is generically finite of degree 2 onto a quadric $Q$. 
These two components intersect in a 39-dimensional locus corresponding to surfaces of the second type where $Q$ is a quadric cone.
 
Griffin \cite{griffin} wrote down explicit equations for the corresponding degeneration of quintic surfaces. These families where later generalized by Reid \cite{reid}, who gave specializations in any dimension $n\ge 2$. For the sake of completeness, we briefly outline his construction.

Let $\PP=\PP(1^{n+2},2,3^2)$ be weighted projective space with homogeneous coordinates $x_0,\ldots,x_{n+1}$ (of degree 1), $y_2$ (of degree 2) and $z_1,z_2$ (of degree 3). Consider the family $\X\subset \PP \times \AA_t^1$ defined by the $4\times 4$ Pfaffians of the skew-symmetric matrix
$$M_t=\begin{pmatrix}
0 & t & y_1 & x_1 & x_2 & z_1 \\
-t& 0 & y_2 & x_2 & x_3 & z_2 \\
-y_1& -y_2 & 0 & z_1 & z_2 & a \\
-x_1 & -x_2 & -z_1 & 0 & tb & b y_1 \\  
-x_2 & -x_3&  -z_2&  -tb&  0 & b y_2 \\   
-z_1  & -z_2&  -a&  -by_1& -by_2  & 0 \\   
\end{pmatrix}
$$where $a=A_5(x_0,\ldots, x_{n+1},y_2)$, $b=B_2(x_0,\ldots,x_{n+1},y_2)$ and $y_1=Y_{1,2}(x_0,\ldots, x_{n+1},y_2)$ are general forms of degree 5, degree 2 and bidegree (1,2) in the $x_i,y_2$ respectively. Then $f:\X\to \AA_t^1$ is a flat family of $n$-folds with $K_{\X | \AA^1}=\O_{\X}(n-3)$. Moreover, for $t\neq 0$, $\X_t$ is isomorphic to a quintic hypersurface, whereas the special fiber $X=\X_0$ is a non-hypersurface, as $|\L|$ is not base point free. In fact, the base locus consists of a single point $x\in X$ in the smooth locus of $X$.  
If $W$ is the blow-up of $X$ at $x$,  the morphism $q: W\to \PP^{n+1}$ is generically a double cover of the rank 3 quadric $x_1x_3=x_2^2$ (compare this with Section \ref{subsec:dimY=3}).

According to \cite{reid}, $X$ has terminal singularities with ordinary double points in dimension $\le n-3$. 
In particular, $X$ is a smooth surface of general type when $n=2$, and a Calabi-Yau threefold with finitely many double points when $n=3$. 

In our proof of Theorem \ref{quintictheorem}, these examples are ruled out in Section \ref{subsec:dimY=3}, via the use of Lemma \ref{lem:decomp}.
That is, we rely on the fact that if $X$ is smooth, then each Weil divisor is Cartier and (at least up to numerical equivalence) a multiple of $\mathcal O_X(1)$, which must fail for the singular examples constructed by Reid above.

\bibliographystyle{plain}

{
}

%
%

{Department of Mathematics, University of Oslo, Box 1053, Blindern, 0316 Oslo, Norway}

{{\it Email:} \verb"johnco@math.uio.no"}

\smallskip

{Mathematisches Institut, Ludwig--Maximilians--Universit\"at M\"unchen, Theresienstr.\ 39, 80333 M\"unchen, Germany}

{{\it Email:} \verb"schreieder@math.lmu.de"}

\end{document}